\documentclass[11pt]{amsart} 

\usepackage{amsmath, amssymb, amsfonts, tikz, hyperref, comment, url, todonotes}
\usepackage{hyperref}
\usepackage{enumitem}
\usepackage{tikz}
\calclayout

\newtheorem{theorem}{Theorem}[section]
\newtheorem{lemma}[theorem]{Lemma}
\newtheorem{proposition}[theorem]{Proposition}
\newtheorem{corollary}[theorem]{Corollary}
\theoremstyle{definition}
\newtheorem{example}[theorem]{Example}
\newtheorem{remark}[theorem]{Remark}
\newtheorem{definition}[theorem]{Definition}
\newtheorem*{definition*}{Definition}
\newtheorem*{theorem*}{Theorem}
\newtheorem*{proposition*}{Proposition}
\newtheorem*{corollary*}{Corollary}

\newtheorem{question}[theorem]{Question}

\renewcommand{\S}{\mathfrak{S}}
\newcommand{\id}{\mathrm{id}}
\newcommand{\Z}{\mathbb{Z}}
\newcommand{\Ind}{\mathrm{Ind}}
\newcommand{\R}{\mathbb{R}}
\newcommand{\Supp}{\mathrm{Supp}}
\newcommand{\htt}{\mathrm{ht}}

\DeclareMathOperator{\spn}{span}

\author{Christian Gaetz}
\thanks{C.G. is supported by a National Science Foundation Graduate Research Fellowship under Grant No. 1122374.}
\address{Department of Mathematics, Massachusetts Institute of Technology, Cambridge, MA 02139}
\email{\href{mailto:gaetz@mit.edu}{{\tt gaetz@mit.edu}}}
\author{Yibo Gao}
\email{\href{mailto:gaoyibo@mit.edu}{{\tt gaoyibo@mit.edu}}}

\begin{document}
\title{Separable Elements in Weyl Groups}
\date{\today}

\begin{abstract}
We define the notion of a \emph{separable} element in a finite Weyl group, generalizing the well-studied class of separable permutations.  We prove that the upper and lower order ideals in weak Bruhat order generated by a separable element are rank-symmetric and rank-unimodal, and that the product of their rank generating functions gives that of the whole group, answering an open problem of Fan Wei.  We also prove that separable elements are characterized by pattern avoidance in the sense of Billey and Postnikov.
\end{abstract}
\maketitle

\section{Introduction and Preliminaries} \label{sec:intro}
A permutation is \emph{separable} if it avoids the patterns 3142 and 2413.  This well-studied class of permutations arose in the study of pop-stack sorting \cite{Avis} and has found applications in algorithmic pattern matching \cite{Bose} and bootstrap percolation \cite{Shapiro}.  These permutations have a remarkable recursive combinatorial structure and are enumerated by the Schr\"{o}der numbers \cite{West}. 

Fan Wei \cite{Wei2012Product} showed that if $w$ is separable, the weak Bruhat intervals $[\id,w]$ and $[w,w_0]$ are rank-symmetric and rank unimodal, and the product of their rank generating functions is $[n]_q !$, the $q$-analog of the factorial of $n$, which is the rank generating function of the weak (or strong) Bruhat order on the symmetric group.  

In this paper we define separable elements of any finite Weyl group $W$ (see Definition \ref{def:separable}) which coincide exactly with separable permutations in the case of the symmetric group.  These elements also have a recursive combinatorial structure and can also be characterized by pattern avoidance, now in the generalized sense of Billey and Postnikov \cite{Billey2005Smoothness} (see Theorem \ref{thm:pattern-classification}).  We solve an open problem of Wei \cite{Wei2012Product} by showing in Theorem \ref{thm:main} that for $w$ separable the weak order intervals $[e,w]$ and $[w,w_0]$ are rank-symmetric and rank-unimodal and that 
\[
\left( \sum_{u \in [e,w]} q^{\ell(u)} \right) \left( \sum_{u \in [w,w_0]} q^{\ell(u)-\ell(w)} \right) = \sum_{u \in W} q^{\ell(u)}.
\]
A second paper \cite{second-separable} establishes further algebraic and combinatorial properties of separable elements.

The remainder of Section \ref{sec:intro} recalls some necessary background on posets, root systems, and Weyl groups.  Section \ref{sec:w0J} discusses the special case of the longest element $w_0^J$ in a parabolic quotient.  Section \ref{sec:defs-and-main-result} introduces the general notion of a \textit{separable} element in an arbitrary finite Weyl group, which coincides with the definition of a separable permutation in type $A$. In Section \ref{sec:defs-and-main-result}, we present our first main result, which answers an open problem of Fan Wei by generalizing Wei's theorem \cite{Wei2012Product} to other types; Section \ref{sec:proof-of-main-thm} contains the proof of this theorem.  Section \ref{sec:pattern-avoidance} states that separable permutations are characterized by pattern avoidance in the sense of Billey and Postnikov \cite{Billey2005Smoothness}, and this Theorem is proven in Section \ref{sec:proof-of-pattern-thm}.

\subsection{Rank functions of posets}
Let $P$ be a finite ranked poset with rank decomposition $P_0\sqcup P_1\sqcup\cdots\sqcup P_r$. Define its \textit{rank generating function}, denoted $F(P)$, to be $F(P):=\sum_{i=0}^r |P_i| \cdot q^i \in \Z[q]$. 

We say that a sequence $a_0,\ldots,a_r$ is \textit{symmetric} if $a_i=a_{r-i}$ for all $i=0,\ldots,r$ and is \textit{unimodal} if there exists $m$ such that 
\[
a_0\leq\cdots\leq a_{m-1}\leq a_m\geq a_{m+1}\geq\cdots\geq a_r.
\]
Similarly, we say that a polynomial $f=\sum_{i=0}^r a_iq^i$ is \textit{symmetric} (sometimes called \textit{palindromic}) if the sequence $a_0,\ldots,a_r\neq0$ is symmetric, and is \textit{unimodal} if this sequence is unimodal. And we say a poset $P$ is \textit{rank-symmetric} if the sequence $|P_0|,|P_1|,\ldots,|P_r|$ is symmetric, and is \textit{rank-unimodal} if this sequence is unimodal.

The following simple lemma can be found in \cite{Stanley1989Log}.
\begin{lemma}\label{lem:sym}
Let $f,g\in\Z[q]$ be two polynomials with nonnegative coefficients. If both are symmetric and unimodal, then $fg$ is symmetric and unimodal.
\end{lemma}

For an element $x$ of $P$, let $V_{x}:=\{y\in P:y\geq x\}$ denote the principal upper order ideal and $\Lambda_x:=\{y\in P:y\leq x\}$ the principal lower order ideal.

\subsection{Root systems and Weyl groups}\label{sub:root}

In the rest of the section, we provide some background on classical theory of root systems and Weyl groups. The reader is referred to \cite{Humphreys1978Introduction} for a detailed exposition. 

Let $\Phi$ be a root system with a chosen set of simple roots $\Delta$ and the corresponding set of positive roots $\Phi^+$, and let $W=W(\Phi)$ be its Weyl group.  Throughout the paper, all roots systems and Weyl groups are assumed to be finite.  Recall the usual (Coxeter) \textit{length} $\ell(w)$ of $w$ is the smallest nonnegative integer $\ell$ such that $w$ can be written as a product of $\ell$ simple transpositions $w=s_{i_1}\cdots s_{i_{\ell}}$.  Such a minimal-length expression for $w$ is called a \emph{reduced expression}.  Any finite Weyl group contains a unique element, denoted $w_0$, of maximum length, called the \emph{longest element}.

For $w\in W(\Phi)$, define its \textit{inversion set} $$I_{\Phi}(w):=\{\alpha\in\Phi^+:w\alpha\in\Phi^-\}.$$ The longest element $w_0$ has inversion set equal to $\Phi^+$.  It is a standard fact that $\ell(w)=|I_{\Phi}(w)|$.  

We will be interested in the \emph{left weak order} $(W, \leq)$ on a Weyl group $W$: $w\leq uw$ if and only if $\ell(w)+\ell(u)=\ell(uw)$. In other words, the left weak order is the transitive closure of the covering relations $w\lessdot s_iw$, where $s_i$ is a simple transposition (reflection across a simple root) and $\ell(w)=\ell(s_iw)-1$. The following propositions are well-known (see for example Proposition 2.1 of \cite{Hohlweg2016On}).
\begin{proposition}\label{prop:weak}
$w\leq u$ if and only if $I_{\Phi}(w)\subseteq I_{\Phi}(u).$
\end{proposition}
\begin{proposition}\label{prop:biconvex}
The inversion set uniquely characterizes an element of the Weyl group. Moreover, $A\subseteq \Phi^+$ is the inversion set of some element if and only if it is \emph{biconvex}; that is, if and only if:
\begin{enumerate}
\item whenever $\alpha,\beta\in A$ and $\alpha+\beta\in\Phi^+$, then $\alpha+\beta\in A$ and,
\item whenever $\alpha,\beta\notin A$ and $\alpha+\beta\in\Phi^+$, then $\alpha+\beta\notin A$.
\end{enumerate}
\end{proposition}

Proposition~\ref{prop:biconvex} allows one to define a \textit{restriction} map. Suppose that $\Delta'\subset \Delta$ is a subset of our simple roots, and let $\Phi'$ be the root system generated by $\Delta'$ and let $(\Phi')^+$ be the corresponding positive roots. Then for $w\in W(\Phi)$, $I_{\Phi}(w)$ is a biconvex set so its restriction $I_{\Phi}(w)\cap (\Phi')^+$ to a smaller dimensional subspace is also biconvex. Let $w|_{\Phi'}\in W(\Phi')$ be the unique element such that $I_{\Phi'}(w|_{\Phi'})=I_{\Phi}(w)\cap (\Phi')^+.$ We call $w|_{\Phi'}$ the \textit{restriction} of $w$ to $\Phi'$.

\begin{example}
Let $e_i$ denote the $i$-th standard basis vector in $\mathbb{R}^n$ and let $\Delta=\{e_1-e_2,e_2-e_3,\ldots,e_6-e_7\}$ generate a root system of type $A_6$ (see Section \ref{sec:irreducible-classification}), whose Weyl group is the symmetric group $\S_7$. Let $w=4623157$ and let $\Delta'=\{e_2-e_3,e_3-e_4\}$ generate $\Phi'$. Recall that for $i<j$ the positive root $e_i-e_j$ is an inversion of $w$ if $w(i)>w(j)$. It can be seen that $I_{\Phi}(w)\cap (\Phi')^+=\{e_2-e_3,e_2-e_4\}$. Therefore, $w|_{\Phi'}=312$; this corresponds to the relative order in which 6,2, and 3 appear in the one-line notation of $w$.
\end{example}

The restriction map was introduced by Billey and Postnikov \cite{Billey2005Smoothness} (there called the \emph{flattening map}) to study smoothness of Schubert varieties. We are here using the restriction map in a more restrictive sense by only considering root subsystems generated by a subset of the simple roots, rather than the more general notion of subsystem considered in that work; see Section \ref{sec:pattern-avoidance} for a discussion of the more general notion.

For convenience, we also define the \textit{root poset}, which is the partial $(\Phi^+, \leq)$ such that $\alpha\leq\beta$ if $\beta-\alpha$ is a nonnegative linear combination of the simple roots. Minimal elements in the root poset are precisely the simple roots $\Delta$.

\subsubsection{The classification of irreducible root systems} \label{sec:irreducible-classification}
Throughout this paper we will refer to the well-known Cartan-Killing classification of irreducible root systems (see, for example, \cite{Humphreys1978Introduction}).  This classification consists of the infinite families of types $A_n, B_n, C_n,$ and $D_n$ as well as the exceptional types $G_2, F_4, E_6, E_7,$ and $E_8$.  Our conventions for realizations of the infinite families are below:
\begin{itemize}
    \item Type $A_n$: $\Phi=\{e_i - e_j \: | \: 1 \leq i,j \leq n+1\} \subset \mathbb{R}^{n+1}$ with positive roots $\Phi^+=\{e_i-e_j \: | \: i<j\}$ and simple roots $\alpha_i=e_i-e_{i+1}$ for $i=1,...,n$.
    \item Type $B_n$: $\Phi=\{\pm e_i \pm e_j, \: \pm e_i \: | \: 1 \leq i,j \leq n\} \subset \mathbb{R}^{n}$ with positive roots $\Phi^+=\{e_i \pm e_j, \: e_i \: | \: i<j\}$ and simple roots $\alpha_i=e_i-e_{i+1}$ for $i=1,...,n-1$ and $\alpha_n=e_n$.
    \item Type $C_n$: $\Phi=\{\pm e_i \pm e_j, \: \pm 2e_i \: | \: 1 \leq i,j \leq n\} \subset \mathbb{R}^{n}$ with positive roots $\Phi^+=\{e_i \pm e_j, \: 2e_i \: | \: i<j\}$ and simple roots $\alpha_i=e_i-e_{i+1}$ for $i=1,...,n-1$ and $\alpha_n=2e_n$.
    \item Type $D_n$: $\Phi=\{\pm e_i \pm e_j \: | \: 1 \leq i,j \leq n\} \subset \mathbb{R}^{n}$ with positive roots $\Phi^+=\{e_i \pm e_j, \: | \: i<j\}$ and simple roots $\alpha_i=e_i-e_{i+1}$ for $i=1,...,n-1$ and $\alpha_n=e_{n-1}+e_n$.
\end{itemize}

When all irreducible components of a root system $\Phi$ are of type $A,D,$ or $E$ (or equivalently, when all roots have the same length), $\Phi$ is said to be \emph{simply laced}.

\section{A Special Case: $w_0^J$} \label{sec:w0J}
Before introducing the general notion of a separable element in a Weyl group, we first consider a special case.

The following facts about parabolic subgroups and quotients are well-known (see, for example, Chapter 2 of \cite{Bjorner-Brenti}).
Let $W$ be a Weyl group, $\Phi$ the corresponding root system, and $\Delta$ a choice of simple roots.  For $J \subseteq \Delta$ let $W_J$ be the \emph{parabolic subgroup} of $W$ generated by the reflections associated to the roots in $J$.  The \emph{parabolic quotient} $W^J \subseteq W$ is a particular choice of coset representatives for $W_J$ defined as follows: $w \in W^J$ if and only if no reduced expression $w=s_{i_1}\cdots s_{i_{\ell}}$ for $w$ has $s_{i_{\ell}} \in J$.  Since $W^J$ forms a complete system of coset representatives, every element $w \in W$ can be uniquely written $w=w^J w_J$ with $w^J \in W^J$ and $w_J \in W_J$; in this decomposition we have
\begin{equation} \label{eq:lengths-add}
    \ell(w)=\ell(w^J) + \ell(w_J).
\end{equation} 
Viewed as a poset under the induced left weak order, $W^J$ is ranked by length.  Thus:
\begin{equation} \label{eq:subgroup-times-quotient}
    F(W)=F(W^J)F(W_J).
\end{equation}
When applying the above decomposition to the longest element $w_0$, we write $w_0=w_0^J w_0(J)$.  It is clear that $w_0(J)$ is the longest element of $W_J$, viewed as a Weyl group in its own right.

The following proposition follows from \cite{Bjorner-Wachs}.

\begin{proposition} \label{prop:quotient-is-ideal}
In the setup from above, let $J \subseteq \Delta$.  Then $W^J$, viewed as a poset under the induced left weak order, is exactly the order ideal $\Lambda_{w_0^J}$ in the left weak order on $W$.
\end{proposition}

\begin{proposition} \label{prop:w0J-case}
For any $J \subseteq \Delta$, the intervals $\Lambda_{w_0^J}$ and $V_{w_0^J}$ are rank-symmetric and rank-unimodal and 
\[
F(\Lambda_{w_0^J})F(V_{w_0^J})=F(W).
\]
\end{proposition}
\begin{proof}
By Proposition \ref{prop:quotient-is-ideal}, $\Lambda_{w_0^J}=W^J$ as ranked posets, so $F(\Lambda_{w_0^J})=F(W^J)$.  Similarly, we have $F(V_{w_0^J})=F(W_J)$ (here the rank function on $V_{w_0^J}$ comes from viewing it as a poset in it's own right, so that $w_0^J$ is in rank zero).  To see this, note that multiplication by $w_0$ is an antiautomorphism of $W$ sending $V_{w_0^J}$ to $\Lambda_{w_0(J)} \cong W_J$.  Since this latter poset, as the weak order on a finite Weyl group, is self-dual, we get the desired result by applying (\ref{eq:subgroup-times-quotient}).

Finally, we note that $F(W^J)$ and $F(W_J)$ are known to be rank-symmetric and rank-unimodal, for example by \cite{Stanley1980Weyl}.
\end{proof}

\section{Separable elements} \label{sec:defs-and-main-result}
Let $\Phi,\Delta,W(\Phi)$ be as in Section~\ref{sub:root}. We start by defining \textit{separable} elements of a Weyl group in a recursive manner, which befits its name ``separable".  A nonrecursive characterization of these elements in terms of pattern avoidance is provided by Theorem \ref{thm:pattern-classification}.

\begin{definition}\label{def:separable}
Let $w\in W(\Phi)$. Then $w$ is \textit{separable} if one of the following holds:
\begin{itemize}
\item $\Phi$ is of type $A_1$;
\item $\Phi=\bigoplus\Phi_i$ is reducible and $w|_{\Phi_i}$ is separable for each $i$;
\item $\Phi$ is irreducible and there exists a \textit{pivot} $\alpha_i\in\Delta$ such that $w|_{\Phi'}\in W(\Phi')$ is separable where $\Phi'$ is generated by $\Delta'=\Delta\setminus\{\alpha_i\}$ and such that either 
\begin{align*}
    & \{\beta\in\Phi^+:\beta\geq\alpha_i\}\subset I_{\Phi}(w), \text{ or} \\
    & \{\beta\in\Phi^+:\beta\geq\alpha_i\}\cap I_{\Phi}(w)=\emptyset.
\end{align*}
\end{itemize}
Since this definition depends not only on $w$ but also on $\Phi$, we sometimes say that $(w,\Phi)$ is separable for clarity.
\end{definition}

\begin{remark}
With some modifications, one could make a similar definition for elements of any finite Coxeter group, a class which includes the finite Weyl groups.  However not much is to be gained by doing this: the only infinite family of finite Coxeter groups which are not Weyl groups are the dihedral groups, and the separable elements in this case are not interesting, they are only the elements $u$ of length zero or one, and their complements $w_0u$.  The only other cases are the two exceptional Coxeter groups of types $H_3$ and $H_4$; these are small enough so that results of the kind that interest us here can easily be investigated by computer.  Therefore we have chosen not to deal with the additional technical complications of working in this only very slightly more general setting.
\end{remark}

This notion of separability is well-defined, as every element in the Weyl group of type $A_1$ (the unique rank-one Weyl group) is separable, and to check whether $w\in W(\Phi)$ is separable, we end up checking a separable condition on root systems with strictly smaller rank. Secondly, this definition suggests a natural way to construct separable elements.
\begin{lemma}\label{lem:construct}
Let $\Phi$ be an irreducible root system with a set of simple roots $\Delta$ and positive roots $\Phi^+$. Let $\alpha_i\in\Delta$, $\Delta':=\Delta\setminus\{\alpha_i\}$, $\Phi'$ be the root subsystem generated by $\Delta'$. Let $A$ be a biconvex set in $(\Phi')^+$. Then $A$ is biconvex in $\Phi^+$ (and dually, $A\cup\{\beta\in\Phi^+:\beta\geq\alpha_i\}$ is biconvex in $\Phi^+$).
\end{lemma}
\begin{proof}
We sketch the proof as the lemma follows easily from classical theory on Dynkin diagrams. Any Dynkin diagram of an irreducible root system is a tree with nodes labeled by the elements of $\Delta$. So removing the node corresponding to $\alpha_i$ from $\Delta$ separates this tree to multiple trees. As a well-known fact (see for example \cite{Humphreys1978Introduction}), we know that for any root $\beta\in\Phi^+$, its support $\{\alpha\in\Delta:\alpha\leq\beta\}$ must form a connected subgraph of the Dynkin diagram. This fact implies that if $\alpha,\beta\in A$ are supported on different connected components of $\Delta'$, then $\alpha+\beta$ cannot be a root. So $A$ is convex. This fact also gives $\Phi^+\setminus(\Phi')^+=\{\beta\in\Phi^+:\beta\geq\alpha_i\}$. Thus, if $\alpha,\beta\notin A$ and $\alpha+\beta\in\Phi^+$, then $\alpha_i\leq \alpha,\beta$ and $\alpha_i\leq\alpha+\beta$ so $\alpha+\beta\notin A$.
\end{proof}

\begin{corollary}
If we start with separable elements in each connected component of $\Delta'=\Delta\setminus\{\alpha_i\}$, then we can obtain two separable elements in $W(\Phi)$ from them by either adding all of $\{\beta\in\Phi^+:\beta\geq\alpha_i\}$ to $I_{\Phi}(w)$ or adding none of it to $I_{\Phi}(w)$.
\end{corollary}

\begin{proposition} \label{prop:symmetries}
Suppose that $(w,\Phi)$ is separable, and let $w_0$ denote the longest element in $W(\Phi)$, then $w_0w$ is also separable.
\end{proposition}
\begin{proof}
It is clear from the definitions that $I_{\Phi}(w_0w)=\Phi^+ \setminus I_{\Phi}(w)$; the proposition then follows immediately from the recursive definition of separable.
\end{proof}

\begin{example}[Type $A_{n-1}$] \label{ex:type-A}
Let us trace through the definition in type $A_{n-1}$ and see that it coincides with separable permutations. Let $\Delta=\{\alpha_1,\ldots,\alpha_{n-1}\}$ where $\alpha_i$ can be represented as $e_i-e_j$ written in coordinate vectors. Let $w\in\S_n$ be a candidate for separable element and let $\alpha_m$ be its pivot. Now, $\Delta'=\{\alpha_1,\ldots,\alpha_{m-1}\}\cup\{\alpha_{m+1},\ldots,\alpha_{n-1}\}$ generates a root system of type $A_{m-1}\oplus A_{n-m-1}$. We need that $w|_{A_{m-1}}$, which is the permutation restricted to the first $m$ indices, and $w|_{A_{n-m-1}}$, which is the permutation restricted to the last $n-m$ indices, are both separable. Moreover, from the condition on $I_{\Phi}(w)$, we also require that either $w(i)>w(j)$ for all $1\leq i\leq m<m+1\leq j\leq n$ or $w(i)<w(j)$ for all $1\leq i\leq m<m+1\leq j\leq n$.  That is, the permutation matrix of $w$ can be recursively constructed by applying the operations of direct sum (placing the matrices of two smaller separable permutations as blocks on the diagonal) and skew sum (placing these blocks on the antidiagonal), starting with the unique permutation of one element as a base case.  This was the original definition of separable permutations, and is easily shown to be equivalent to avoiding 3142 and 2413 (see, for example, Lemma 2.3 of \cite{Wei2012Product}).
\end{example}
\begin{example}[Type $B_2$]\label{ex:B2}
Assume $\Delta=\{\alpha_1,\alpha_2\}$ where $\alpha_1=e_1-e_2$ and $\alpha_2=e_2$. Then $\Phi^+=\{\alpha_1,\alpha_2,\alpha_1+\alpha_2,\alpha_1+2\alpha_2\}$. Let us write down all possible separable elements by their inversion sets as suggested in Lemma~\ref{lem:construct}. If we choose $\alpha_1$ as a pivot, then $\Delta'=\{\alpha_2\}$ has rank 1 so $I_{\Phi'}(w|_{\Phi'})$ is either $\emptyset$ or $\{\alpha_2\}$. Now to obtain a separable elements, we either add to $I_{\Phi}(w)$ all positive roots supported on $\alpha_1$, or add nothing. We then end up with $I_{\Phi}(w)$ equal to one of $\emptyset$, $\{\alpha_2\}$, $\{\alpha_1,\alpha_1+\alpha_2,\alpha_1+2\alpha_2\}$ or $\Phi^+$. Similarly, if we choose the pivot at $\alpha_2$, we end up with $I_{\Phi}(w)$ equal to one of $\emptyset$, $\{\alpha_1\}$, $\{\alpha_2,\alpha_1+\alpha_2,\alpha_1+2\alpha_2\}$ or $I_{\Phi}(w)$. They are shown in Figure~\ref{fig:B2} via inversion sets.

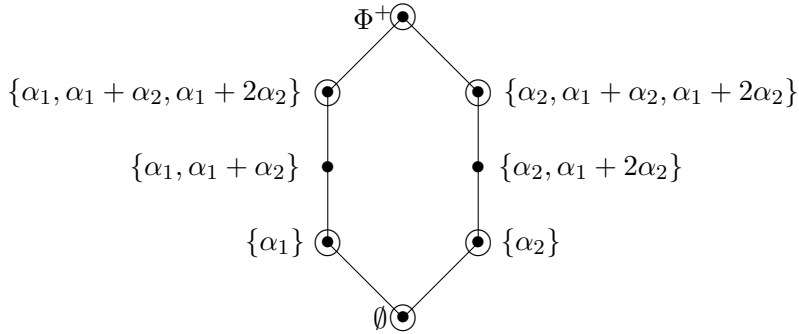
\begin{figure}[ht]
\centering
\begin{tikzpicture}
\node at (0,0){$\bullet$};
\node at (0,0){$\bigcirc$};
\node at (-1,1){$\bullet$};
\node at (-1,1){$\bigcirc$};
\node at (-1,2){$\bullet$};
\node at (-1,3){$\bullet$};
\node at (-1,3){$\bigcirc$};
\node at (1,1){$\bullet$};
\node at (1,1){$\bigcirc$};
\node at (1,2){$\bullet$};
\node at (1,3){$\bullet$};
\node at (1,3){$\bigcirc$};
\node at (0,4){$\bullet$};
\node at (0,4){$\bigcirc$};
\draw (0,0)--(-1,1)--(-1,3)--(0,4)--(1,3)--(1,1)--(0,0);
\node at (-1.7,1){$\{\alpha_1\}$};
\node at (-2.5,2){$\{\alpha_1,\alpha_1+\alpha_2\}$};
\node at (-3.3,3){$\{\alpha_1,\alpha_1+\alpha_2,\alpha_1+2\alpha_2\}$};
\node at (1.7,1){$\{\alpha_2\}$};
\node at (2.5,2){$\{\alpha_2,\alpha_1+2\alpha_2\}$};
\node at (3.3,3){$\{\alpha_2,\alpha_1+\alpha_2,\alpha_1+2\alpha_2\}$};
\node at (-0.3,0){$\emptyset$};
\node at (-0.4,4){$\Phi^+$};
\end{tikzpicture}
\caption{Weak order of type $B_2$ labeled by inversion sets, where separable elements are circled.}
\label{fig:B2}
\end{figure}
\end{example}

\begin{example}
For any root system $\Phi$ and any $J \subseteq \Delta$ the elements $w_0(J)$ and $w_0^J$ discussed in Section \ref{sec:w0J} are separable.  For $w_0(J)$, all elements in $\Delta \setminus J$ are pivots, and, once these are removed, the restriction of $w_0(J)$ to each irreducible component will be the longest element of the corresponding Weyl group, which is clearly separable.  For $w_0^J$ a similar argument applies, or one can note that $w_0^J=w_0w_0(J)$ and apply Proposition \ref{prop:symmetries}.  Theorem \ref{thm:main} below greatly generalizes Proposition \ref{prop:w0J-case} to all separable elements.
\end{example}

We are now ready to state our first main theorem, which generalizes a result of Fan Wei \cite{Wei2012Product} for separable permutations, answering the open problem of how to extend this result to other Weyl groups.

\begin{theorem}\label{thm:main}
Let $w\in W(\Phi)$ be separable. Then the weak order upper order ideal $V_{w}$ and lower order ideal $\Lambda_w$ are both rank symmetric and rank unimodal. In addition,
\[
F(V_w)F(\Lambda_w)=F(W(\Phi)).
\]
\end{theorem}

In light of Theorem \ref{thm:main}, it makes sense to ask whether the posets $\Lambda_w$ and $V_w$ are \emph{strongly Sperner} for $w$ separable.  A poset is strongly Sperner if for all $k=1,2,...$ no union of $k$ antichains is larger than the union of the largest $k$ ranks.  A poset which is rank-symmetric, rank-unimodal, and strongly Sperner is called \emph{Peck}.  Parabolic quotients in the strong order \cite{Stanley1980Weyl} and the whole weak order in type $A$ \cite{sperner-paper} are known to be Peck. 

\begin{question}
Let $(w,\Phi)$ be separable; are $\Lambda_w$ and $V_w$ strongly Sperner?
\end{question}

\section{Proof of Theorem \ref{thm:main}} \label{sec:proof-of-main-thm}
We need a few lemmas to start with. 
\begin{lemma}\label{lem:restrictProduct}
Let $\Phi$, $\Delta$, $\Delta'\subset\Delta$, $\Phi'$ be as in Section~\ref{sub:root} and let $w\in W(\Phi)$. Take a simple root $\alpha_i\in\Delta$ and let $s_{\alpha_i}$ be the corresponding simple transposition. Assume $\ell(ws_{\alpha_i})=\ell(w)+1$. Then
\begin{enumerate}
\item if $\alpha_i\in\Delta'$, $(ws_{\alpha_i})|_{\Phi'}=(w|_{\Phi'})s_{\alpha_i}$ and $\ell((w|_{\Phi'})s_{\alpha_i})=\ell(w|_{\Phi'})+1$;
\item if $\alpha_i$ is orthogonal to all roots in $\Delta'$, that is, if $\alpha_i\notin\Delta'$ and there are no edges between $\alpha_i$ and $\Delta'$ in the Dynkin diagram, $(ws_{\alpha_i})|_{\Phi'}=w|_{\Phi'}.$
\end{enumerate}
\end{lemma}
\begin{proof}
In general, for $u\in W(\Phi)$ and $\alpha_i\in\Delta$, if $\ell(us_{\alpha_i})=\ell(u)+1$, then $I_{\Phi}(us_{\alpha_i})=s_{\alpha_i}I_{\Phi}(u)\cup\{\alpha_i\}$ (see \cite{Humphreys1978Introduction}). We also know that $s_{\alpha_i}$ permutes $\Phi^+\setminus\{\alpha_i\}$ and sends $\alpha_i$ to $-\alpha_i$. 

We show two Weyl group elements are the same by comparing their inversion sets.  For (1), 
\begin{align*}
    I_{\Phi'}\big((ws_{\alpha_i}|_{\Phi'})\big)&=I_{\Phi}(ws_{\alpha_i})\cap(\Phi')^+\\
    &=\big(s_{\alpha_i}I_{\Phi}(w)\cup\{\alpha_i\}\big)\cap(\Phi')^+\\
    &=\big(s_{\alpha_i}I_{\Phi}(w)\cap(\Phi')^+\big)\cup\{\alpha_i\},
\end{align*}
where the last step follows from the fact that $\alpha_i\in(\Phi')^+$. Recall also that $s_{\alpha_i}$ permutes $(\Phi')^+\setminus\{\alpha_i\}$. Since $\alpha_i\notin s_{\alpha_i}I_{\Phi}(w)$, we have 
\begin{align*}
    s_{\alpha_i}I_{\Phi}(w)\cap(\Phi')^+&=s_{\alpha_i}I_{\Phi}(w)\cap s_{\alpha_i}(\Phi')^+\\
    &=s_{\alpha_i}\big(I_{\Phi}(w)\cap(\Phi')^+\big)\\
    &=s_{\alpha_i}I_{\Phi'}(w|_{\Phi'}).
\end{align*}
Therefore, $I_{\Phi'}\big((ws_{\alpha_i}|_{\Phi'})\big)=s_{\alpha_i}I_{\Phi'}(w|_{\Phi'})\cup\{\alpha_i\}$. Since $\alpha_i\notin I_{\Phi}(w)$, we have that $\alpha_i\notin I_{\Phi'}(w|_{\Phi'})$ and so $\ell((w|_{\Phi'})s_{\alpha_i})=\ell(w|_{\Phi'})+1$. With this, $s_{\alpha_i}I_{\Phi'}(w|_{\Phi'})\cup\{\alpha_i\}=I_{\Phi'}((w|_{\Phi'})s_{\alpha_i})$. We can thus conclude as desired that $(ws_{\alpha_i})|_{\Phi'}=(w|_{\Phi'})s_{\alpha_i}$.

For (2), we similarly observe that 
\begin{align*}
    I_{\Phi'}\big((ws_{\alpha_i}|_{\Phi'})\big)&=\big(s_{\alpha_i}I_{\Phi}(w)\cup\{\alpha_i\}\big)\cap(\Phi')^+\\
    &=s_{\alpha_i}I_{\Phi}(w)\cap(\Phi')^+,
\end{align*}
as $\alpha_i\notin(\Phi')^+$.  Since $\alpha_i$ is orthogonal to all roots in $\Phi'$, $s_{\alpha_i}$ fixes every root in $(\Phi')^+$. Thus, 
\begin{align*}
    s_{\alpha_i}I_{\Phi}(w)\cap(\Phi')^+ &=s_{\alpha_i}I_{\Phi}(w)\cap s_{\alpha_i}(\Phi')^+ \\
    &=s_{\alpha_i}\big(I_{\Phi}(w)\cap(\Phi')^+\big) \\
    &=I_{\Phi}(w)\cap(\Phi')^+=I_{\Phi'}(w|_{\Phi'}).
\end{align*}
Therefore $(ws_{\alpha_i})|_{\Phi'}=w|_{\Phi'}$. 
\end{proof}

\begin{lemma}\label{lem:notdepend}
Let $\Phi$, $\Delta$, $\Delta'\subset\Delta$, $\Phi'$ be as in Section~\ref{sub:root}. Fix $w'\in W(\Phi')$. Then
$$\sum_{w\in W(\Phi),\ w|_{\Phi'}=w'}q^{\ell(w)-\ell(w')}=\frac{F(W(\Phi))}{F(W(\Phi'))}$$
where $\ell$ denotes the lengths in respective Weyl groups. Moreover, this is a symmetric and unimodal polynomial.
\end{lemma}
\begin{proof}
By definition, $I_{\Phi}(w)\supset I_{\Phi'}(w|_{\Phi'})$ so $\ell(w)\geq\ell(w|_{\Phi'})$. This means our left hand side is indeed a polynomial. For $w\in W(\Phi')$, let $\Ind(w'):=\{w\in W(\Phi):w|_{\Phi'}=w'\}$. Suppose $\ell(w's_{\alpha_i})=\ell(w')+1$ for some $\alpha_i\in\Delta'$. We claim that there is a bijection $\varphi:\Ind(w')\rightarrow\Ind(w's_{\alpha_i})$ via $w\mapsto ws_{\alpha_i}$ such that $\ell(ws_{\alpha_i})=\ell(w)+1$. This follows immediately from Lemma~\ref{lem:restrictProduct}. Namely, for $w\in\Ind(w')$, $(ws_{\alpha_i})|_{\Phi'}=(w|_{\Phi'})s_{\alpha_i}=w's_{\alpha_i}$ so $ws_{\alpha_i}\in\Ind(w's_{\alpha_i})$. The inverse map is given in the same way. As for the length, if $\ell(w)=\ell(ws_{\alpha_i})+1$, by using Lemma~\ref{lem:restrictProduct} on $ws_{\alpha_i}$, we find $\ell(w')=\ell(w's_{\alpha_i})+1$, which is a contradiction. Thus, we must have $\ell(ws_{\alpha_i})=\ell(w)+1$ instead.

To interpret this bijection in another way, let $f_{w'}:=\sum_{w\in\Ind(w')}q^{\ell(w)-\ell(w')}$ be the left hand side of the equation in the lemma statement. Then $f_{w'}=f_{w's_{\alpha_i}}$. Since $s_{\alpha_i}$'s generate $W(\Phi')$ for $\alpha_i\in\Delta'$, $f_{w'}=f$ is constant on $W(\Phi')$. As a result,
\begin{align*}
F(W(\Phi))=&\sum_{w\in W(\Phi)}q^{\ell(w)}=\sum_{w'\in W(\Phi')}\sum_{w\in\Ind(w')}q^{\ell(w)}\\
=&\sum_{w'\in W(\Phi')}q^{\ell(w')}\sum_{w\in\Ind(w')}q^{\ell(w)-\ell(w')}\\
=&\sum_{w'\in W(\Phi')}q^{\ell(w')}\cdot f=f\cdot F(W(\Phi')).
\end{align*}
This means $f_{w'}=f=F(W(\Phi))/F(W(\Phi'))$ as desired. And this shows $f$ is indeed a polynomial.

To show that $f$ is symmetric and unimodal, let us consider the special case where $w'=\id$. By definition, $\Ind(\id_{\Phi'})$ consists of $w\in W(\Phi)$ such that $w|_{\Phi'}=\id_{\Phi'}$. These are precisely the set of $w$'s which do not have inversions at $\alpha_j$ for all $\alpha_j\in\Delta'$. In other words, the set $\Ind(\id_{\Phi'})$ equals the parabolic quotient $W(\Phi)^{\Delta'}$. By Proposition \ref{prop:w0J-case}:
$$f=f_{\id_{\Phi'}}=\sum_{w\in\Ind(\id_{\Phi'})}q^{\ell(w)}=F(W^{\Delta'})$$
is symmetric and unimodal.
\end{proof}

We are now ready to finish the proof of Theorem \ref{thm:main}.

\begin{proof}[Proof of Theorem~\ref{thm:main}]
Unsurprisingly as Definition~\ref{def:separable} suggests, we are going to proceed by induction on the rank of $\Phi$. The base case is $\Phi=A_1$, whose Weyl group consists of two elements $\id<w_0$ and both are separable. Clearly the results hold.

Now assume $\Phi=\bigoplus\Phi_i$ is reducible. Then $W(\Phi)=\prod W(\Phi_i)$ as groups and as posets. Therefore, $F(V_w)=\prod F(V_{w|_{\Phi_i}})$ and $F(\Lambda_w)=\prod F(\Lambda_{w|_{\Phi_i}})$. By induction hypothesis and Lemma~\ref{lem:sym}, both are symmetric and unimodal. Moreover, $$F(V_w)F(\Lambda_w)=\prod_i F(V_{w|_{\Phi_i}})F(\Lambda_{w|_{\Phi_i}})=\prod_i F(W(\Phi_i))=F(W(\Phi))$$.

The key case is when $\Phi$ is irreducible. Let the pivot be $\alpha_i\in\Delta$. As before, let $\Delta'=\Delta\setminus\{\alpha_i\}$ which generates a root subsystem $\Phi'$ of $\Phi$. There are two analogous cases. First assume that $\{\beta\in\Phi^+:\beta>\alpha_i\}\cap I_{\Phi}(w)=\emptyset$. This also means $I_{\Phi}(w)\subset(\Phi')^+$. By characterization of Weyl group elements using inversion sets, we obtain $F(\Lambda_w)=F(\Lambda_{w|_{\Phi'}})$. For the upper order ideal $V_w$, we know $u\in V_w$ if and only if $I_{\Phi}(u)\supset I_{\Phi}(w)$. But $I_{\Phi}(w)\subset(\Phi')^+$. Therefore, $I_{\Phi}(u)\supset I_{\Phi}(w)$ if and only if $$I_{\Phi'}(u|_{\Phi'})=I_{\Phi}(u)\cap (\Phi')^+\supset I_{\Phi}(w)\cap (\Phi')^+=I_{\Phi'}(w|_{\Phi'}),$$ 
which is equivalent to $u|_{\Phi'}\geq w|_{\Phi'}$. By Lemma~\ref{lem:notdepend}, where $\Ind(u')=\{u\in W(\Phi):u|_{\Phi'}=u'\}$ for $u'\in W(\Phi')$,
\begin{align*}
F(V_w)=&\sum_{u\geq w}q^{\ell(u)}=\sum_{u'\geq w|_{\Phi'}}\sum_{u\in\Ind(u')}q^{\ell(u)}\\
=&\sum_{u'\geq w|_{\Phi'}}q^{\ell(u')}\sum_{u\in\Ind(u')}q^{\ell(u)-\ell(u')}\\
=&\sum_{u'\geq w|_{\Phi'}}q^{\ell(u')}\frac{F(W(\Phi))}{F(W(\Phi'))}\\
=&F(V_{w|_{\Phi'}})\frac{F(W(\Phi))}{F(W(\Phi'))}.
\end{align*}
Here, by Lemma~\ref{lem:notdepend}, we see that both $F(V_w)$ and $F(\Lambda_w)$ are symmetric and unimodal. Moreover, 
\begin{align*}
    F(V_w)F(\Lambda_w)&=F(V_{w|_{\Phi'}})F(\Lambda_{w|_{\Phi'}})\frac{F(W(\Phi))}{F(W(\Phi'))}\\
    &=F(W(\Phi'))\frac{F(W(\Phi))}{F(W(\Phi'))}\\
    &=F(W(\Phi))
\end{align*}
by the induction hypothesis. In the alternative case where we instead have $\{\beta\in\Phi^+:\beta>\alpha_i\}\subset I_{\Phi}(w)$, we obtain $F(V_w)=F(V_{w|_{\Phi'}})$ and $F(\Lambda_w)=F(\Lambda_{w|_{\Phi'}})\frac{F(W(\Phi))}{F(W(\Phi'))}$, and the same conclusion can be reached.
\end{proof}

\section{Pattern avoidance} \label{sec:pattern-avoidance}

We saw in Example~\ref{ex:type-A} that separable elements in type $A_{n-1}$ are exactly those permutations avoiding the patterns 3142 and 2413.  Billey and Postnikov \cite{Billey2005Smoothness} introduced a notion of pattern avoidance in general Weyl groups.  In this section we will see that separable elements in general are characterized by pattern avoidance.

Let $\Phi$ be a root system spanning a real vector space $V$, with positive roots $\Phi^+$.  A subset $\Phi' \subset \Phi$ is a \emph{subsystem} of $\Phi$ if $\Phi'=\Phi \cap U$ for some linear subspace $U \subset V$ (this notion generalizes the \emph{parabolic} subsystems considered earlier, where $\Phi'$ is generated by a subset of the simple roots).  It is clear that any such $\Phi'$ is itself a root system.  For $w \in W(\Phi)$, we say $w$ \emph{contains the pattern} $(w',\Phi')$ if $I_{\Phi}(w)\cap U = I_{\Phi'}(w')$ (and we extend our earlier notation by writing $w|_{\Phi'}=w'$ in this case).  We say $w$ avoids $(w',\Phi')$ if it does not contain any pattern isomorphic to $(w',\Phi')$.

\begin{lemma} \label{lem:complement-pattern}
Let $w_0(\Phi)$ and $w_0(\Psi)$ denote the longest elements in $W(\Phi)$ and $W(\Psi)$ respectively, then $(w,\Phi)$ contains the pattern $(u, \Psi)$ if and only if $(w_0(\Phi)w, \Phi)$ contains $(w_0(\Psi)w, \Psi)$.
\end{lemma}
\begin{proof}
Let $U \subset \spn_{\R}(\Phi)$ be a linear subspace such that $\Psi=\Phi\cap U$ and $u=w|_{\Psi}$, meaning that $I_{\Psi}(u)=I_{\Phi}(w) \cap U$.  Then, 
\begin{align*}
    I_{\Psi}(w_0(\Psi)u)&=\Psi \setminus I_{\Psi}(u) \\
    &=(\Phi \cap U) \setminus (I_{\Phi}(w) \cap U) \\
    &=(\Phi \setminus I_{\Phi}(w)) \cap U \\
    &=I_{\Phi}(w_0(\Phi)w) \cap U
\end{align*}
And so $w_0(\Phi)w$ contains $w_0(\Psi)u$ as a pattern.  The converse is similar.
\end{proof}

Lemma \ref{lem:complement-pattern} is relevant to separable elements, as the forbidden patterns will be shown to be closed under left multiplication by $w_0$.  For example, notice that $4321 \cdot 3142 = 2413$.

\begin{proposition} \label{prop:characterized-by-avoidance}
If $(w,\Phi)$ is separable and $(u, \Psi)$ occurs as a pattern in $(w,\Phi)$, then $(u, \Psi)$ is separable.  Thus there exists a set $\mathcal{P}$ of patterns, independent of $W$, such that an element $w$ in any finite Weyl group $W(\Phi)$ is separable if and only if $w$ avoids all patterns in $\mathcal{P}$.
\end{proposition}
\begin{proof}
Suppose without loss of generality that $\Phi$ is irreducible, and let $w \in W(\Phi)$ be separable with pivot $\alpha \in \Delta$.  Let $\Delta', \Delta''$ be the connected components of the Dynkin diagram on $\Delta \setminus \{\alpha\}$ so that 
\[
\Delta = \Delta' \sqcup \{\alpha\} \sqcup \Delta'',
\]
and let $\Phi', \Phi''$ be the root systems generated by $\Delta'$ and $\Delta''$ (it is possible that removing $\alpha$ leaves one, two, or three connected components in the Dynkin diagram; the other cases are exactly analogous).  Let $w'=w|_{\Phi'}$ and $w''=w|_{\Phi''}$; by the definition of separable there are two cases:
\begin{enumerate}
    \item $I_{\Phi}(w) = I_{\Phi'}(w') \sqcup \emptyset \sqcup I_{\Phi''}(w'')$ with $w',w''$ separable, or 
    \item $I_{\Phi}(w) = I_{\Phi'}(w') \sqcup \{\beta \in \Phi \: | \: \beta \geq \alpha\} \sqcup I_{\Phi''}(w'')$ with $w',w''$ separable.
\end{enumerate}
Suppose first that we are in case (1) and suppose $(u, \Phi \cap U)$ appears as a pattern in $(w, \Phi)$.  This means that 
\begin{align*}
I_{\Phi \cap U}(u)&= I_{\Phi}(w) \cap U \\
&= (I_{\Phi'}(w') \cap U) \sqcup (I_{\Phi''}(w'') \cap U) \\
&= I_{\Phi' \cap U}(w'|_{\Phi' \cap U}) \sqcup I_{\Phi'' \cap U}(w''|_{\Phi'' \cap U})
\end{align*}
Now, since $\Phi'=\Phi \cap \spn_{\R}(\Delta')$ and $\Phi''=\Phi \cap \spn_{\R}(\Delta'')$, we see that $w'|_{\Phi' \cap U}$ and $w''|_{\Phi'' \cap U}$ occur as patterns in $w'$ and $w''$ respectively.  By induction on the rank, we may assume that for $(v,\Psi)$ separable and of smaller rank than $(w,\Phi)$ any pattern occuring in $(v,\Psi)$ is also separable.  Applying this to $w'$ and $w''$ we conclude that $w'|_{\Phi' \cap U}$ and $w''|_{\Phi'' \cap U}$ are separable, and therefore that $u$ is separable.

In case (2), suppose $(u,\Phi \cap U)$ appears as a pattern in $w$.  This implies that $w_0(\Phi \cap U)u$ appears as a pattern in $w_0(\Phi)w$ by Lemma \ref{lem:complement-pattern}.  By Proposition~\ref{prop:symmetries}, $w_0(\Phi)w$ is separable because $w$ is, and furthermore this element falls under case (1), so that we may conclude that $w_0(\Phi \cap U)u$ is separable by the preceding argument.  One more application of Proposition~\ref{prop:symmetries} then implies that $u$ is separable, completing the proof.

The set $\mathcal{P}$ is the set of all minimal non-separable elements in all finite Weyl groups (under the pattern containment order).
\end{proof}

In fact, it is possible to completely classify the patterns in $\mathcal{P}$; it turns out that $\mathcal{P}$ is finite, which is not clear a priori from Proposition \ref{prop:characterized-by-avoidance}.  Remarkably, in the simply-laced case there are no new patterns to consider.

\begin{theorem} \label{thm:pattern-classification}
Let $\mathcal{P}$ be the set of minimal non-separable patterns, as in Proposition \ref{prop:characterized-by-avoidance}.  Then $\mathcal{P}$ consists of:
\begin{enumerate}[label=\roman*]
    \item The patterns of type $A_3$ given by the permutations 3142 and 2413,
    \item the two patterns of length two in type $B_2$ (see Figure \ref{fig:B2}), and
    \item the six patterns of lengths two, three, and four in type $G_2$.
\end{enumerate}
In particular, an element $(w,\Phi)$ in a root system of simply-laced type is separable if and only if it avoids the type $A_3$ patterns 3142 and 2413.
\end{theorem}

Theorem \ref{thm:pattern-classification} is proven in Section \ref{sec:proof-of-pattern-thm}; this proof necessarily requires some type by type checking, although all simply-laced types are considered together.

\section{Proof Theorem \ref{thm:pattern-classification}} \label{sec:proof-of-pattern-thm}
The proof of Theorem \ref{thm:pattern-classification} which occupies this section is divided into several cases, depending on the type of $W(\Phi)$.  In each case we show that any element avoiding the patterns from $\mathcal{P}$ has a pivot $\alpha_i$ as in Definition \ref{def:separable}, and is thus separable by induction on rank.  Since patterns of separable elements are separable by Proposition \ref{prop:characterized-by-avoidance}, this implies that separable elements are exactly those which avoid the patterns from $\mathcal{P}$.

We keep the same setting as above: $\Phi$ is a root system with a choice of positive roots $\Phi^+$ which determines the simple roots $\Delta=\{\alpha_1,\ldots,\alpha_n\}$, and $W(\Phi)$ is the Weyl group of $\Phi$.  Recall that $\Delta$ forms a basis for the ambient vector space. 

\begin{definition}\label{def:small}
For a positive root $\alpha\in\Phi^+$, write it as the uniquely linear combination of simple roots $\alpha=\sum_{i=1}^n c_i\alpha_i$. Then the \textit{support} $\Supp(\alpha)$ of $\alpha$ is $\{\alpha_i|c_i>0\}$. In addition, we say that $\alpha$ is \textit{small}, if $c_i\in\{0,1\}$ for all $i=1,\ldots,n$.
\end{definition}

Note that all roots are small in type $A$, since $e_i-e_j=\sum_{k=i}^{j-1} (e_i-e_{i+1})$ and the roots $e_i-e_{i+1}$ are simple.

Recall the previously mentioned classical fact that the support of any positive root in an irreducible root system forms a connected subgraph of the Dynkin diagram of $\Phi$.  Notice also that all simple roots are small, and, moreover, small roots of an irreducible root system are in bijection with connected subgraphs of its Dynkin diagram (see, e.g. Chapter VI Section 1 of \cite{Bourbaki}).  

By Proposition \ref{prop:symmetries} and Lemma \ref{lem:complement-pattern}, all patterns in $\mathcal{P}$ come in pairs: $u$ and $w_0u$.  Since $I_{\Phi}(w_0w)=\Phi^+ \setminus I_{\Phi}(w)$, and since the biconvexity condition in Proposition \ref{prop:biconvex} is symmetric with respect to switching the roles of inversions and noninversions, many arguments in what follows carry through just as well with these roles reversed.  To highlight this symmetry and for conciseness, given a Weyl group element $w$, we will say that a positive root $\alpha$ is colored \textit{black} if $\alpha\in I_{\Phi}(w)$ and is colored \textit{white} if $\alpha\notin I_{\Phi}(w)$. The symmetry of swapping the roles of these two colors corresponds to complementing inversion sets, that is, to replacing $w$ with $w_0w$.  This terminology also allows us the convenience of saying that several roots \emph{have the same color}, meaning that they are all white or all black.

\subsection{The simply-laced case} \label{sec:simply-laced}
In this section, let $\Phi$ be a simply-laced irreducible root system. Let $w\in W(\Phi)$ avoid patterns 3142 and 2413. The goal is to find a pivot of $w$ (see Definition~\ref{def:separable}). Then induction will complete the proof. We do this in two steps. First, we show that all small roots (Definition~\ref{def:small}) admit a candidate pivot, in a sense to be made precise later. Next, we go up in the root poset to argue that the color of a root depends only on its support. Throughout the argument, the following lemma plays a key role.
\begin{lemma}\label{lem:simply-ind-main}
Let $\Phi$ be a simply-laced irreducible root system and let $w\in W(\Phi)$ avoid 3142 and 2413. Assume $\alpha,\beta,\gamma\in\Phi^+$ span a root system of type $A_3$ as simple roots (in this order), i.e. $(\alpha,\beta)=(\beta,\gamma)=-1$, $(\alpha,\gamma)=0$ where $(,)$ denote the inner product that comes with the root system. Then if $\beta$, $\alpha+\beta$ and $\beta+\gamma$ are colored the same (w.r.t $w$), then $\alpha+\beta+\gamma$ has the same color as well.
\end{lemma}
\begin{proof}
Without loss of generality assume that $\beta$, $\alpha+\beta$ and $\beta+\gamma$ are black, meaning that all of them are in $I_{\Phi}(w)$. If $\alpha$ is black, then $\alpha+(\beta+\gamma)$ is black by biconvexity. Similarly if $\gamma$ is black, then $(\alpha+\beta)+\gamma$ is black. Now we assume $\alpha$ and $\gamma$ are white. If $\alpha+\beta+\gamma$ is white as well, then $w|_{\Phi'}$ is exactly 2413, where $\Phi'$ is the root subsystem generated by $\alpha,\beta,\gamma$, which is a contradiction. Therefore in all cases, $\alpha+\beta+\gamma$ is colored black. 
\end{proof}

\begin{lemma}\label{lem:simply-step1}
Let $\Phi$ be a simply-laced irreducible root system and $w\in W(\Phi)$ avoid 3142 and 2413. Then there exists a simple root $\alpha_t\in\Delta$, such that all small roots whose supports contain $\alpha_t$ are colored the same as $\alpha_t$.
\end{lemma}
\begin{proof}
We proceed by induction on the rank of $\Phi$. The case that $\Phi$ is of type $A$ is well-known (see e.g. \cite{Wei2012Product}). Now suppose $\Phi$ is of type $D$ or type $E$ of rank $n$. In all cases, the Dynkin diagram of $\Phi$ can be viewed as a chain of length $n-1$ (a root subsystem of type $A_{n-1}$), together with another leaf $\alpha_n$. Let $\Phi'$ be the root subsystem generated by $\alpha_1,\ldots,\alpha_{n-1}$, which is irreducible of type $A_{n-1}$. By the induction hypothesis, $w|_{\Phi'}$ avoids 3142 and 2413 so it is separable. Let $P=\{\alpha_{j_1},\ldots,\alpha_{j_k}\}$, $k\geq1$, be the set of all possible pivots of $w|_{\Phi'}$. Since these pivots must have the same sign as $\alpha_1+\cdots+\alpha_{n-1}$, let us assume that they are all colored black. 

Let $\alpha_m$ be the simple root that is connected to $\alpha_n$ in the Dynkin diagram. We know that $\alpha_m$ is the unique vertex in the Dynkin diagram with valence 3.

If $\alpha_n$ is colored black, then we claim that all small roots supported on $\alpha_{j}$ are colored black, for any $\alpha_j\in P$. We already know that all small roots supported on $\alpha_{j}$ but not on $\alpha_n$ are black, and for those small roots $\alpha$ supported on both $\alpha_{j}$ and $\alpha_n$, we see that $\alpha-\alpha_n$ is a small root supported on $\alpha_{j}$ but not on $\alpha_n$, which is black. By biconvexity, $\alpha=(\alpha-\alpha_n)+\alpha_n$ must be black. From now on, assume that $\alpha_n$ is colored white. The data we get so far can be seen in Figure~\ref{fig:simply-step1-topwhite}.
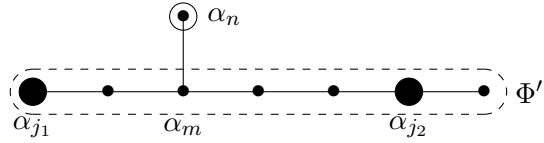
\begin{figure}[ht]
\centering
\begin{tikzpicture}[scale=1.0]
\node[draw,shape=circle,fill=black,scale=1.0,label=below:{$\alpha_{j_1}$}] at (0,0) {};
\node at (1,0) {$\bullet$};
\node[label=below:{$\alpha_m$}] at (2,0) {$\bullet$};
\node at (3,0) {$\bullet$};
\node at (4,0) {$\bullet$};
\node[draw,shape=circle,fill=black,scale=1.0,label=below:{$\alpha_{j_2}$}] at (5,0) {};
\node at (6,0) {$\bullet$};
\node[draw,shape=circle,fill=white,scale=1.0,label=right:{$\alpha_{n}$}] at (2,1) {};
\node at (2,1) {$\bullet$};
\draw(0,0)--(6,0);
\draw(2,1)--(2,0);
\draw[dashed](0,0.3)arc(90:270:0.3);
\draw[dashed](6,-0.3)arc(-90:90:0.3);
\draw[dashed](0,0.3)--(6,0.3);
\draw[dashed](0,-0.3)--(6,-0.3);
\node at (6.6,0) {$\Phi'$};
\end{tikzpicture}
\caption{A simply-laced irreducible root system with some coloring}
\label{fig:simply-step1-topwhite}
\end{figure}

Since the Dynkin diagram is a tree, for any two simple roots $\alpha_i$ and $\alpha_j$, there exists a unique path that connects them. Write $r(\alpha_i,\alpha_j)=r(\alpha_j,\alpha_i)$ as the unique small root whose support is exactly this path.

If $r(\alpha_{j},\alpha_n)$ is black, for some $\alpha_j\in P$, then we claim that all small roots supported on $\alpha_j$ are black. Again, it suffices to restrict our attention to those small roots supported on both $\alpha_n$ and $\alpha_j$. Without loss of generality, assume that $\alpha_j$ is on the right side of $\alpha_m$ (or is $\alpha_m$). If the small root of interest is of the form $r(\alpha_j,\alpha_n)+\gamma$, where $\gamma$ is a small root corresponding to a path on the right of $\alpha_j$, then we can let $\alpha=r(\alpha_j,\alpha_n)-\alpha_j$, $\beta=\alpha_j$. In this setting, $\alpha,\beta,\gamma$ form the simple roots of a type $A_3$ subsystem. Moreover, $\beta$ and $\beta+\alpha_j$ are black since $\beta$ is a pivot in $\Phi'$ and $\alpha+\beta$ is black by assumption. The conditions of Lemma~\ref{lem:simply-ind-main} are satisfied, and thus $\alpha+\beta+\gamma=r(\alpha_j,\alpha_n)+\gamma$ is black. Similarly, if the small root of interest is of the form $r(\alpha_j,\alpha_n)+\alpha$, where $\alpha$ is a small root whose support is a path to the left of $\alpha_m$, then we can let $\beta=r(\alpha_j,\alpha_n)-\alpha_n$ and $\gamma=\alpha_n$. Again, conditions of Lemma~\ref{lem:simply-ind-main} are satisfied, so $\alpha+\beta+\gamma=r(\alpha_j,\alpha_n)+\alpha$ is black. These two cases are shown in Figure~\ref{fig:rblack}.
\begin{figure}[ht]
\centering
\begin{tikzpicture}[scale=1.0]
\node at (0,0) {$\bullet$};
\node at (1,0) {$\bullet$};
\node at (2,0) {$\bullet$};
\node at (3,0) {$\bullet$};
\node at (4,0) {$\bullet$};
\node[draw,shape=circle,fill=black,scale=1.0,label=below:{$\alpha_{j}$}] at (5,0) {};
\node at (6,0) {$\bullet$};
\node[draw,shape=circle,fill=white,scale=1.0] at (2,1) {};
\node at (2,1) {$\bullet$};
\draw(0,0)--(6,0);
\draw(2,1)--(2,0);

\draw[dashed](2.3,1)arc(0:180:0.3);
\draw[dashed](4,-0.3)arc(-90:90:0.3);
\draw[dashed](1.7,0)arc(180:270:0.3);
\draw[dashed](1.7,0)--(1.7,1);
\draw[dashed](2,-0.3)--(4,-0.3);
\draw[dashed](2.3,1)--(2.3,0.3)--(4,0.3);
\draw[dashed](5.3,0)arc(0:360:0.3);
\draw[dashed](6.3,0)arc(0:360:0.3);
\node at (2.5,0.5) {$\alpha$};
\node at (5,0.5) {$\beta$};
\node at (6,0.5) {$\gamma$};
\end{tikzpicture}
\begin{tikzpicture}[scale=1.0]
\node at (0,0) {$\bullet$};
\node at (1,0) {$\bullet$};
\node at (2,0) {$\bullet$};
\node at (3,0) {$\bullet$};
\node at (4,0) {$\bullet$};
\node[draw,shape=circle,fill=black,scale=1.0,label=below:{$\alpha_{j}$}] at (5,0) {};
\node at (6,0) {$\bullet$};
\node[draw,shape=circle,fill=white,scale=1.0] at (2,1) {};
\node at (2,1) {$\bullet$};
\draw(0,0)--(6,0);
\draw(2,1)--(2,0);

\draw[dashed](0,0.3)arc(90:270:0.3);
\draw[dashed](1,-0.3)arc(-90:90:0.3);
\draw[dashed](2,0.3)arc(90:270:0.3);
\draw[dashed](5,-0.3)arc(-90:90:0.3);
\draw[dashed](2.3,1)arc(0:360:0.3);
\draw[dashed](0,0.3)--(1,0.3);
\draw[dashed](0,-0.3)--(1,-0.3);
\draw[dashed](2,0.3)--(5,0.3);
\draw[dashed](2,-0.3)--(5,-0.3);
\node at (0.5,0.5){$\alpha$};
\node at (2.5,1){$\gamma$};
\node at (3.5,0.4){$\beta$};
\end{tikzpicture}
\caption{The usage of Lemma~\ref{lem:simply-ind-main} in the case where some $r(\alpha_j,\alpha_n)$ is black.}
\label{fig:rblack}
\end{figure}
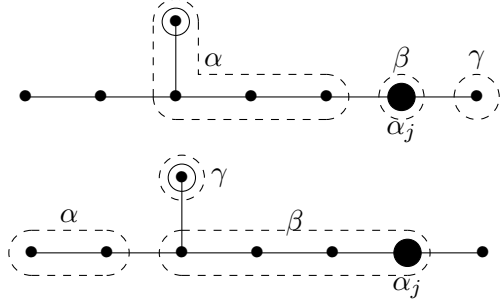
Finally, in general, if the small root of interest has the form $\alpha+r(\alpha_j,\alpha_n)+\gamma$, where $\alpha$ corresponds to a path to the left of $\alpha_m$ and $\gamma$ corresponds to a path to the right of $\alpha_j$, we can let $\beta=r(\alpha_j,\alpha_n)$ and apply Lemma~\ref{lem:simply-ind-main} in the same way. As a result, all small roots supported on $\alpha_j$ are black.

The final remaining case allows us to assume that $r(\alpha_j,\alpha_n)$ is white for all $\alpha_j\in P$. We first notice that $\alpha_n+\alpha_m$ is white. If not, then $\alpha_m\notin P$ and take any $\alpha_j\in P$ and assume without loss of generality that $\alpha_j$ is on the right of $\alpha_m$. Let $\alpha_k$ be the simple root that is directly on the right side of $\alpha_m$. Then $r(\alpha_j,\alpha_n)=(\alpha_n+\alpha_m)+r(\alpha_k,\alpha_j)$, which must be black, a contradiction. This means $\alpha_n+\alpha_m$ is white. Let $\Phi''$ be the root subsystem of type $A_{n-1}$, that consists of $\Delta'':=\Delta\setminus\{\alpha_m,\alpha_n\}\cup\{\alpha+n+\alpha_m\}$ as a base. Intuitively, we merge $\alpha_n$ and $\alpha_m$ together. By induction hypothesis, $w|_{\Phi''}$ avoids 3142 and 2413 so it is separable. We claim that its pivot is white. It suffices to show that the union of supports of all white roots of $\Phi''$ is $\Delta''$. Then it suffices to show that for any black (simple) root in $\Delta''$, it is smaller than some white root in $\Phi''$. Take such a black root in $\Delta''$, as $\alpha_n+\alpha_m$ is white, we can assume that such a root is $\alpha_i\in\Delta\setminus\{\alpha_n,\alpha_m\}$. If $\alpha_i\in P$, then $r(\alpha_i,\alpha_n)$ is white and is supported on $\alpha_i$. If $\alpha_i\notin P$, by definition of $P$, there exists a white root $r$ in $\Phi'$ that is supported on $\alpha_i$. Either $r$ is a root in $\Phi''$ or $r+\alpha_n$ is a root in $\Phi''$ and it is white in both cases. As a result, we know that the pivot of $\Phi''$ is white.

If we know that $\alpha_n+\alpha_m$ is a pivot of $\Phi''$, then every small root of $\Phi$ supported on $\alpha_n$ is either $\alpha_n$ itself, or some small root of $\Phi''$, which are all white. So in the end, we are going to derive a contradiction, assuming $\alpha_n+\alpha_m$ is not a pivot of $\Phi''$. Let $\alpha_k$ be a white pivot of $\Phi''$, that lies somewhere on the right side of $\alpha_m$, without loss of generality. By definition of $P$, for any $\alpha_j\in P$, roots $r(\alpha_j,\alpha_k)$ are black, so $r(\alpha_j,\alpha_k)$ cannot be a positive root of $\Phi''$, meaning that $\alpha_j$ lies weakly to the left of $\alpha_m$. See Figure~\ref{fig:rwhite} for a possible scenario. 
\begin{figure}[ht]
\centering
\begin{tikzpicture}[scale=1.0]
\node at (0,0) {$\bullet$};
\node[draw,shape=circle,fill=black,scale=1.0,label=below:{$\alpha_{j_1}$}] at (1,0) {};
\node[draw,shape=circle,fill=black,scale=1.0,label=below:{$\alpha_{j_2}$}] at (2,0) {};
\node at (3,0) {$\bullet$};
\node at (4,0) {$\bullet$};
\node[draw,shape=circle,fill=white,scale=1.0,label=below:{$\alpha_{k}$}] at (5,0) {};
\node at (5,0) {$\bullet$};
\node at (6,0) {$\bullet$};
\node[draw,shape=circle,fill=white,scale=1.0,label=right:{$\alpha_n$}] at (2,1) {};
\node at (2,1) {$\bullet$};
\draw(0,0)--(6,0);
\draw(2,1)--(2,0);

\draw[dashed](2.3,1)arc(0:180:0.3);
\draw[dashed](1.7,0)arc(180:360:0.3);
\draw[dashed](1.7,0)--(1.7,1);
\draw[dashed](2.3,0)--(2.3,1);
\node at (3,0.5){$\alpha_n+\alpha_m$};
\end{tikzpicture}
\caption{The case where $r(\alpha_j,\alpha_n)$ are white for all $\alpha_j\in P$.}
\label{fig:rwhite}
\end{figure}
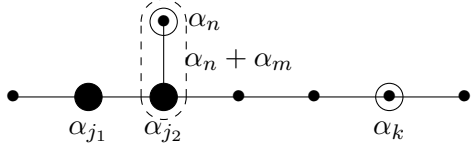

If $\alpha_n+\alpha_m$ is not a pivot for $w|_{\Phi''}$, then there exists $\alpha_q\in\Delta\setminus\{\alpha_n,\alpha_m\}$ such that $r(\alpha_n,\alpha_q)$ is black. We know from assumption that $\alpha_q\notin P$. Moreover, $\alpha_q$ must lie strictly to the left of $\alpha_k$ as $\alpha_k$ is a white pivot of $\Phi''$. We have the following two cases:

\noindent\textbf{Case 1:} $\alpha_q$ is on the left of $\alpha_m$. Take any $\alpha_j\in P$. If $\alpha_j$ is not in the support of $r(\alpha_n,\alpha_q)$, or in other words, $\alpha_j$ is on the left of $\alpha_q$, then $r(\alpha_j,\alpha_q)-\alpha_q$ is a black small root. But $r(\alpha_j,\alpha_n)=(r(\alpha_j,\alpha_q)-\alpha_q)+r(\alpha_n,\alpha_q)$ is black, a contradiction. This means $\alpha_j$ lies between $\alpha_q$ and $\alpha_m$. Let $\beta=r(\alpha_q,\alpha_m)$, $\alpha=\alpha_n$ and $\gamma=r(\alpha_m,\alpha_k)-\alpha_m$ (Figure~\ref{fig:rwhite-qleft}). Then $\alpha,\beta,\gamma$, as simple roots, span a root subsystem of type $A_3$. Moreover, $\alpha+\beta=r(\alpha_q,\alpha_n)$ is black, $\beta$ and $\beta+\gamma$ are black because they are in $\Phi'$ and are supported on $\alpha_j$. This means $\alpha+\beta+\gamma$ is black by Lemma~\ref{lem:simply-ind-main}, contradicting the fact that $\alpha_k$ is a white pivot of $\Phi''$.
\begin{figure}[ht]
\centering
\begin{tikzpicture}[scale=1.0]
\node at (0,0) {$\bullet$};
\node at (0,-0.45) {$\alpha_q$};
\node[draw,shape=circle,fill=black,scale=1.0,label=below:{$\alpha_{j}$}] at (1,0) {};
\node at (2,0) {$\bullet$};
\node at (2,-0.45) {$\alpha_m$};
\node at (3,0) {$\bullet$};
\node at (4,0) {$\bullet$};
\node[draw,shape=circle,fill=white,scale=1.0,label=below:{$\alpha_{k}$}] at (5,0) {};
\node at (5,0) {$\bullet$};
\node at (6,0) {$\bullet$};
\node[draw,shape=circle,fill=white,scale=1.0,label=right:{$\alpha_n$}] at (2,1) {};
\node at (2,1) {$\bullet$};
\draw(0,0)--(6,0);
\draw(2,1)--(2,0);

\draw[dashed](0,0.3)arc(90:270:0.3);
\draw[dashed](2,-0.3)arc(-90:90:0.3);
\draw[dashed](0,0.3)--(2,0.3);
\draw[dashed](0,-0.3)--(2,-0.3);
\draw[dashed](2.3,1)arc(0:360:0.3);
\draw[dashed](3,0.3)arc(90:270:0.3);
\draw[dashed](5,-0.3)arc(-90:90:0.3);
\draw[dashed](3,0.3)--(5,0.3);
\draw[dashed](3,-0.3)--(5,-0.3);
\node at (0.5,0.5) {$\beta$};
\node at (1.5,1) {$\alpha$};
\node at (4,0.5) {$\gamma$};
\end{tikzpicture}
\caption{The case where $\alpha_q$ is on the left of $\alpha_m$, and $r(\alpha_j,\alpha_n)$ are white for all $\alpha_j\in P$.}
\label{fig:rwhite-qleft}
\end{figure}
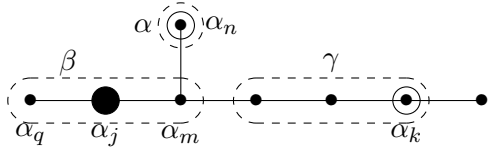

\noindent\textbf{Case 2:} $\alpha_q$ is on the right of $\alpha_m$. We already know that $\alpha_q$ is on the left of $\alpha_k$. Similarly as above, take any $\alpha_j\in P$. Let $\alpha=\alpha_n$, $\beta=r(\alpha_j,\alpha_q)$, $\gamma=r(\alpha_q,\alpha_k)-\alpha_q$. We see that $\beta$ and $\beta+\gamma$ are black since they are small roots in $\Phi'$ supported on $\alpha_j$, and $\alpha+\beta$ is black since it equals $(r(\alpha_j,\alpha_m)-\alpha_m)+r(\alpha_n,\alpha_q)$, which is either $r(\alpha_n,\alpha_q)$ or the sum of two black roots. By Lemma~\ref{lem:simply-ind-main}, $\alpha+\beta+\gamma$ is black, contradicting $\alpha_k$ being the white pivot in $\Phi''$. This case is shown in Figure~\ref{fig:rwhite-qright}.
\begin{figure}[ht]
\centering
\begin{tikzpicture}[scale=1.0]
\node at (0,0) {$\bullet$};
\node at (3,-0.45) {$\alpha_q$};
\node[draw,shape=circle,fill=black,scale=1.0,label=below:{$\alpha_{j}$}] at (1,0) {};
\node at (2,0) {$\bullet$};
\node at (2,-0.45) {$\alpha_m$};
\node at (3,0) {$\bullet$};
\node at (4,0) {$\bullet$};
\node[draw,shape=circle,fill=white,scale=1.0,label=below:{$\alpha_{k}$}] at (5,0) {};
\node at (5,0) {$\bullet$};
\node at (6,0) {$\bullet$};
\node[draw,shape=circle,fill=white,scale=1.0,label=right:{$\alpha_n$}] at (2,1) {};
\node at (2,1) {$\bullet$};
\draw(0,0)--(6,0);
\draw(2,1)--(2,0);

\draw[dashed](1,0.3)arc(90:270:0.3);
\draw[dashed](3,-0.3)arc(-90:90:0.3);
\draw[dashed](1,0.3)--(3,0.3);
\draw[dashed](1,-0.3)--(3,-0.3);
\draw[dashed](2.3,1)arc(0:360:0.3);
\draw[dashed](4,0.3)arc(90:270:0.3);
\draw[dashed](5,-0.3)arc(-90:90:0.3);
\draw[dashed](4,0.3)--(5,0.3);
\draw[dashed](4,-0.3)--(5,-0.3);
\node at (0.5,0.5) {$\beta$};
\node at (1.5,1) {$\alpha$};
\node at (4.5,0.5) {$\gamma$};
\end{tikzpicture}
\caption{The case where $\alpha_q$ is on the right of $\alpha_m$, and $r(\alpha_j,\alpha_n)$ are white for all $\alpha_j\in P$.}
\label{fig:rwhite-qright}
\end{figure}
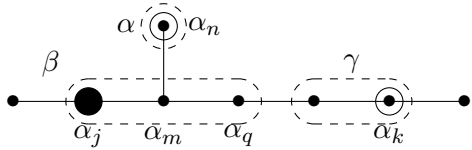
All cases are finished.
\end{proof}

Lemma~\ref{lem:simply-step1} deals with all small roots (Definition~\ref{def:small}). From there on, we will utilize the following lemma to show by induction that the color of a root depends only on its support.
\begin{lemma}\label{lem:simply-step2}
Let $\Phi$ be a simply laced irreducible root system that is not of type $A$. Let $r\in\Phi^+$ be a root of full support, i.e. $\Supp(r)=\Delta$, and let $\alpha_i\in\Delta$ be an arbitrary simple root. Then at least one of the following is true:
\begin{enumerate}
    \item $r=\beta_1+\beta_2$, where $\beta_1,\beta_2\in\Phi^+$ whose support contain $\alpha_i$;
    \item $r=\alpha+\beta+\gamma$, where $\alpha,\beta,\gamma\in\Phi^+$, $(\alpha,\beta)=(\beta,\gamma)=-1$, $(\alpha,\gamma)=0$ (as in Lemma~\ref{lem:simply-ind-main}), and $\alpha_i\in\Supp(\beta)$. 
\end{enumerate}
\end{lemma}
\begin{proof}
The proof is case by case. For the exceptional types $E_6$, $E_7$ or $E_8$, the claim has been verified by computer. If $\Phi$ is of type $D_n$, $n\geq4$, we do a case check as follows. Let $\Delta=\{\alpha_1,\ldots,\alpha_n\}$, where $\alpha_1$ is the unique vertex in the Dynkin diagram of valence 3, whose three branches consist of $\{\alpha_2,\ldots,\alpha_{n-2}\}$, $\{\alpha_{n-1}\}$ and $\{\alpha_n\}$. See Figure~\ref{fig:typeD}.
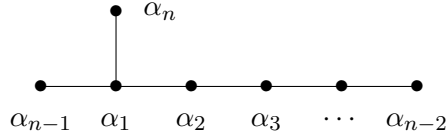
\begin{figure}[ht]
\centering
\begin{tikzpicture}
\draw(0,0)--(5,0);
\draw(1,0)--(1,1);
\node[label=below:{$\alpha_{n-1}$}] at (0,0)
{$\bullet$};
\node[label=below:{$\alpha_{1}$}] at (1,0)
{$\bullet$};
\node[label=below:{$\alpha_{2}$}] at (2,0)
{$\bullet$};
\node[label=below:{$\alpha_{3}$}] at (3,0)
{$\bullet$};
\node[label=below:{$\cdots$}] at (4,0)
{$\bullet$};
\node[label=below:{$\alpha_{n-2}$}] at (5,0)
{$\bullet$};
\node[label=right:{$\alpha_{n}$}] at (1,1)
{$\bullet$};
\end{tikzpicture}
\caption{A labeling for the Dynkin diagram of type $D_n$}
\label{fig:typeD}
\end{figure}

Let $r_0=\alpha_1+\cdots+\alpha_n$, which is small. If $r=r_0$, then if $\alpha_i=\alpha_n$ (or similarly $\alpha_{n-1}$), take $\alpha=\alpha_{n-1}$, $\beta=\alpha_1+\alpha_n$, $\gamma=\alpha_2+\cdots+\alpha_{n-2}$, which satisfy (2). If $\alpha_i\in\{\alpha_1,\ldots,\alpha_{n-2}\}$, take $\alpha=\alpha_n$, $\beta=\alpha_1+\cdots+\alpha_{n-2}$ and $\gamma=\alpha_{n-1}$ analogously. For the more general case, we can write $r=r_0+(\alpha_1+\cdots+\alpha_k)$, where $k\leq n-3$.

\noindent\textbf{Case 1:} $\alpha_i\in\{\alpha_1,\ldots,\alpha_k\}$. Let $\beta_1=r$ and $\beta_2=\alpha_1+\cdots+\alpha_k$ so (1) is satisfied.

\noindent\textbf{Case 2:} $\alpha_i=\alpha_n$ (or similarly $\alpha_{n-1}$). Let $\alpha=\alpha_{n-1}+\alpha_1+\cdots+\alpha_{k}$, $\beta=\alpha_n+\alpha_1+\cdots+\alpha_{k-1}$, $\gamma=\alpha_k+\alpha_{k+1}+\cdots+\alpha_{n-2}$ so (2) is satisfied.

\noindent\textbf{Case 3:} $\alpha_i\in\{\alpha_{k+1},\ldots,\alpha_{n-2}\}$. Let $\alpha=\alpha_{n-1}+\alpha_1+\cdots+\alpha_k$, $\gamma=\alpha_n+\alpha_1+\cdots+\alpha_k$, $\beta=\alpha_{k+1}+\cdots+\alpha_{n-2}$ so (2) is satisfied.
\end{proof}
\begin{remark}
Lemma~\ref{lem:simply-step2} is not true in type $A_{n-1}$.  Taking $r$ to be the highest root $e_1-e_n$, it is easy to verify that no decomposition of the kinds described in the Lemma exists.
\end{remark}

We are now ready to finish the pattern classification for the simply-laced root systems.
\begin{proof}[Proof of Theorem~\ref{thm:pattern-classification} for irreducible simply-laced root systems]
The type $A$ case is well-known (see \cite{Wei2012Product} for example). Let $\Phi$ be an irreducible simply-laced root system that is not of type $A$. By Lemma~\ref{lem:simply-step1}, there exists a simple root $\alpha_t$ such that all small roots supported on $\alpha_t$ have the same color. Assume it is black without loss of generality. Recall that the height $\htt(\alpha)$ of a positive root $\alpha=\sum c_i\alpha_i$ is $\sum c_i$. We now use induction on $\htt(r)$ to show that if $r$ is supported on $\alpha_t$, then $r$ is black. We can assume $r$ is not small. Then $\Supp(r)$ is simply-laced and is not of type $A$. Therefore, by Lemma~\ref{lem:simply-step2}, there are two cases. First, there exists $\beta_1,\beta_2\in\Phi^+$ such that $r=\beta_1+\beta_2$ and the supports of $\beta_1,\beta_2$ both contain $\alpha_t$. By induction hypothesis, both $\beta_1$ and $\beta_2$ are black since they have strictly smaller height so by biconvexity, $r$ is black. Second, there exists $\alpha,\beta,\gamma$ which span a root subsystem of type $A_3$ such that $\alpha+\beta+\gamma=r$ and $\alpha_t\in\Supp(\beta)$. By induction hypothesis on the height, $\beta$, $\alpha+\beta$ and $\beta+\gamma$ are all black so by Lemma~\ref{lem:simply-ind-main}, $r=\alpha+\beta+\gamma$ is black as well. As a result, the inductive step goes through so we are done.
\end{proof}
\begin{remark}
The proof presented in this section partially unifies type $D$ and $E$. It is possible to provide a more direct argument for type $D$ but checking the exceptional type $E_8$ via a search through all of its Weyl group elements is computationally infeasible. We therefore take an intermediate step by Lemma~\ref{lem:simply-step1}; checking Lemma~\ref{lem:simply-step2} for the exceptional types is much more feasible since it only requires searching through the root systems and not the Weyl groups.
\end{remark}

\subsection{The non-simply-laced cases} \label{sec:not-simply-laced}

\subsubsection{Type $B_n$} \label{sec:type-B}

Let $\Phi$ denote the root system of type $B_n$; it will be convenient to argue in terms of the root poset $Q$ of $\Phi$ which is shown for $n=4$ in Figure \ref{fig:B4}.  Let $Q_A$ be the order ideal generated by $\sum_i \alpha_i$ in $Q$; that is, $Q_A$ is the induced subposet of $Q$ on all of the small roots.  This poset is isomorphic to the root poset of type $A_n$ under the map induced by $\alpha_i \mapsto e_i-e_{i+1}$.  For a root $\beta=\sum_i c_i \alpha_i$, let $\widehat{\beta}=\sum_{i: c_i \neq 0} \alpha_i$ be the small root with the same support as $\beta$.

Now suppose $w \in W(\Phi)$ avoids the patterns from Theorem \ref{thm:pattern-classification} (i) and (ii) (no patterns of type $G_2$ may occur because the ratio of the lengths of the long and short roots is $\sqrt{2}$ in type $B_n$ and $\sqrt{3}$ in type $G_2$).  We wish to show that $w$ has a pivot $\alpha \in \Delta$.

\begin{lemma} \label{lem:same-color-as-hat}
For any $\beta \in \Phi^+$, the roots $\beta, \widehat{\beta}$ have the same color.
\end{lemma}
\begin{proof}
If $\beta \in Q_A$, then $\widehat{\beta}=\beta$ so the claim is trivial, so assume otherwise.  In this case 
\begin{align*}
    \beta &= \alpha_i+ \cdots \alpha_{j-1}+2(\alpha_j + \cdots + \alpha_n) \\
    \widehat{\beta} &= \alpha_i+ \cdots \alpha_{j-1}+\alpha_j + \cdots + \alpha_n
\end{align*}
for some $i\leq j < n$.  Then the subsystem spanned by $\beta,\widehat{\beta}$ is 
\[
\Psi^+ = \{\beta, \widehat{\beta}, \beta-\widehat{\beta}, 2\widehat{\beta}-\beta\} 
\]
with simple roots $\beta-\widehat{\beta}$ and $2\widehat{\beta}-\beta$.  Now suppose $\beta$ and $\widehat{\beta}$ have different colors; without loss of generality, $\beta$ is white and $\widehat{\beta}$ is black.  By biconvexity, we must have that $\beta-\widehat{\beta}$ is white and $2\widehat{\beta}-\beta$ is black.  But then $(w|_{\Psi}, \Psi)$ is a pattern of type $B_2$ with two inversions, violating our assumption that $w$ avoided such patterns.  Thus $\beta,\widehat{\beta}$ have the same color.
\end{proof}

By Lemma \ref{lem:same-color-as-hat}, the colors of the elements of $Q \setminus Q_A$ are determined by those of $Q_A$.  Thus it suffices to find $\alpha \in \Delta$ such that all elements of $\{\beta \in Q_A \: | \: \beta \geq \alpha \}$ have the same color, and this $\alpha$ will be a pivot.

Let $\Phi_A$ be the root system of type $A_n$; it is clear that the isomorphism of posets $\phi: Q_A \to \Phi_A^+$ given by $\alpha_i \mapsto e_i-e_{i+1}$ preserves the property of a set of positive roots being biconvex, so to every $u \in W(\Phi)$ there exists a unique element $u_A \in W(\Phi_A)$ such that $I_{\Phi_A}(u_A)=\phi(I_{\Phi}(u) \cap Q_A)$.  Note that $\Phi_A$ is \emph{not} a subsystem of $\Phi$ and thus $u_A$ does not occur as a pattern in $(u,\Phi)$.  However, the next lemma shows that $u_A$ does behave like a pattern in $u$ with respect to the avoidance of the type $A_3$ patterns in question.

\begin{lemma}
If $w \in W(\Phi)$ avoids the patterns from Theorem \ref{thm:pattern-classification} (i) and (ii), then so does $w_A \in W(\Phi_A)$.
\end{lemma}
\begin{proof}
Suppose without loss of generality that $(w_A, \Phi_A)$ contains the pattern 3142 (the argument for 2413 is obtained by reversing the colors; $w_A$ can never contain the type $B_2$ patterns), but that $(w,\Phi)$ avoids all four patterns.  This means that there exist $1 \leq i<j<k<\ell \leq n+1$ such that $e_i-e_j, e_k - e_{\ell},$ and $e_i-e_{\ell}$ are inversions of $w_A$ and $e_j-e_k, e_i-e_k,$ and $e_j-e_{\ell}$ are noninversions.  If $\ell<n+1$ then all of these are contained in the type $A_{n-1}$ subsystem of both $\Phi_A$ and $\Phi$, meaning that $w$ also contains 3142, a contradiction, so assume $l=n+1$.  In this case, the definition of $w_A$ implies that $e_i-e_j, e_k,$ and $e_i$ are inversions of $w$ and $e_i-e_k, e_j-e_k,$ and $e_j$ are not.  By Lemma \ref{lem:same-color-as-hat}, this means that $e_j+e_k$ is a noninversion, since $\widehat{e_j+e_k}=e_j$ and $e_i+e_j$ is an inversion, since $\widehat{e_i+e_j}=e_i$.  But this violates biconvexity: $(e_i-e_k)+(e_j+e_k)=e_i+e_j$ but the first two are noninversions while the last is an inversion.
\end{proof}

Now, since we have assumed $(w,\Phi)$ avoids the forbidden patterns, so does $w_A$.  By the type $A$ case of the Theorem, $w_A$ has a pivot $\gamma$, and by construction (recalling Lemma \ref{lem:same-color-as-hat}) $\phi^{-1}(\gamma)$ is a pivot for $w$.  The subsystems of $\Phi$ which remain after removing this pivot are of type $A$ or of type $B_k$ for $k<n$, and so $w$ is separable as desired.

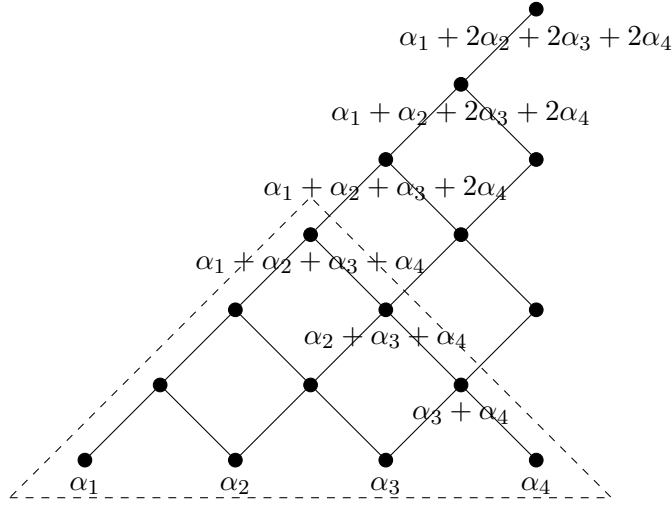
\begin{figure}[ht]
    \begin{tikzpicture}
    \node[draw,shape=circle,fill=black,scale=0.5](a)[label=below: {$\alpha_1$}] at (-4,0) {};
    \node[draw,shape=circle,fill=black,scale=0.5](b)[label=below: {$\alpha_2$}] at (-2,0) {};
    \node[draw,shape=circle,fill=black,scale=0.5](c)[label=below: {$\alpha_3$}] at (0,0) {};
    \node[draw,shape=circle,fill=black,scale=0.5](d)[label=below: {$\alpha_4$}] at (2,0) {};
    
    \node[draw,shape=circle,fill=black,scale=0.5](e) at (-3,1) {};
    \node[draw,shape=circle,fill=black,scale=0.5](f) at (-1,1) {};
    \node[draw,shape=circle,fill=black,scale=0.5](g)[label=below: {$\alpha_3+\alpha_4$}] at (1,1) {};
    
    \node[draw,shape=circle,fill=black,scale=0.5](h) at (-2,2) {};
    \node[draw,shape=circle,fill=black,scale=0.5](i)[label=below: {$\alpha_2+\alpha_3+\alpha_4$}] at (0,2) {};
    \node[draw,shape=circle,fill=black,scale=0.5](j) at (2,2) {};
    
    \node[draw,shape=circle,fill=black,scale=0.5](k) at (-1,3)[label=below: {$\alpha_1+\alpha_2+\alpha_3+\alpha_4$}] {};
    \node[draw,shape=circle,fill=black,scale=0.5](l) at (1,3) {};
    
    \node[draw,shape=circle,fill=black,scale=0.5](m)[label=below: {$\alpha_1+\alpha_2+\alpha_3+2\alpha_4$}] at (0,4) {};
    \node[draw,shape=circle,fill=black,scale=0.5](n) at (2,4) {};
    
    \node[draw,shape=circle,fill=black,scale=0.5](o)[label=below: {$\alpha_1+\alpha_2+2\alpha_3+2\alpha_4$}] at (1,5) {};
    
    \node[draw,shape=circle,fill=black,scale=0.5](p)[label=below: {$\alpha_1+2\alpha_2+2\alpha_3+2\alpha_4$}] at (2,6) {};
    
    \draw (a)--(e)--(b)--(f)--(c)--(g)--(d);
    \draw (e)--(h)--(f)--(i)--(g)--(j);
    \draw (h)--(k)--(i)--(l)--(j);
    \draw (k)--(m)--(l)--(n);
    \draw (m)--(o)--(n);
    \draw (o)--(p);
    
    \draw[dashed] (-5,-.5)--(3,-.5)--(-1,3.5)--(-5,-.5);
    
    \end{tikzpicture}
    \caption{The root poset $Q$ for type $B_4$; the type-$A_4$-like subset $Q_A$ is enclosed in dashed lines.}
    \label{fig:B4}
\end{figure}

\subsubsection{Type $C_n$}

Let $\Phi_B$ and $\Phi_C$ denote the root systems of types $B_n$ and $C_n$ respectively, viewed as subsets of the same vector space $V$; the type $C_n$ case will be proven by reduction to the type $B_n$ case, which is discussed in Section \ref{sec:type-B}.  

The roots of $\Phi_B$ and $\Phi_C$ differ only by scaling, so $W(\Phi_B)=W(\Phi_C)$ as subgroups of $GL(V)$; let $W$ denote this group.  Let $\tau: \Phi_B \to \Phi_C$ be defined by 
\begin{align*}
    \pm e_i \pm e_j &\mapsto \pm e_i \pm e_j \\
    \pm e_i &\mapsto \pm 2e_i
\end{align*}
for all $1 \leq i<j \leq n$.  It is clear from the definition that $\tau(I_{\Phi_B}(w))=I_{\Phi_C}(w)$ for $w \in W$.

Now, suppose that $(w,\Phi_C)$ avoids the patterns from Theorem \ref{thm:pattern-classification} (i) and (ii); we wish to show that $w$ has a pivot with respect to $\Phi_C$.

\begin{lemma} \label{lem:C-avoids-implies-B-avoids}
If $(w,\Phi_C)$ avoids the patterns from Theorem \ref{thm:pattern-classification} (i) and (ii) then so does $(w,\Phi_B)$.
\end{lemma}
\begin{proof}
Let $U$ be a subspace of $V$ and suppose that $(w|_U, \Phi_B \cap U)$ is a type $A_3$ pattern.  Then all roots in $\Phi_B \cap U$ have the same length and no pair is orthogonal, so they must all be of the form $\pm e_i \pm e_j$.  Thus $\Phi_C \cap U = \Phi_B \cap U$ and $(w, \Phi_C)$ also contains the pattern.

Next suppose that $(w|_U, \Phi_B \cap U)$ is one of the forbidden type $B_2$ patterns.  Up to some immaterial signs, a type $B_2$ subsystem of $\Phi_B$ has positive roots $\Psi_B^+=\{e_i, e_j, e_i+e_j,e_i-e_j\}$ for some $i,j$.  Thus $I_{\Phi_B}(w) \cap U$ contains exactly two elements from $\Psi_B^+$.  Now, $\tau(\Psi_B^+)=\{2e_i, 2e_j, e_i+e_j, e_i-e_j\}$ is the set of positive roots $\Psi_C^+$ of a type $B_2$ subsystem of $\Phi_C$.  But then $I_{\Phi_C}(w) \cap U = \tau(I_{\Phi_B}(w) \cap U)$ also has size two, meaning that $(w,\Phi_C)$ also contains the pattern.
\end{proof}

By Lemma \ref{lem:C-avoids-implies-B-avoids} $(w,\Phi_B)$ avoids the forbidden patterns, and so by the result of Section \ref{sec:type-B}, $w$ has a pivot with respect to $\Phi_B$.  This means that for some $\alpha \in \Delta_B$ all roots in the dual order ideal $\mathcal{I}=\{ \beta \in \Phi_B \: | \: \beta \geq_B \alpha\}$ have the same color.  It is straightforward to see that $\tau(\mathcal{I})$ is the dual order ideal in the root poset of $\Phi_C$ generated by the simple root $\tau(\alpha)$, and so $w$ also has a pivot with respect to $\Phi_C$.  Thus, since the parabolic subsystems which remain after removing $\tau(\alpha)$ are of type $A$ or of type $C_k$ with $k<n$, we conclude that $w$ is separable.   

\subsubsection{Type $F_4$}
Let $\Phi$ be the root system of type $F_4$.  In order to conclude that there are no minimal non-separable elements $(w,\Phi)$, it suffices to show that any element of $W(\Phi)$ avoiding the two patterns of type $B_2$ from Theorem~\ref{thm:pattern-classification}(ii) is in fact separable.  There are 1152 elements of $W(\Phi)$ and 18 subsystems of $\Phi^+$ isomorphic to $B_2$.  It has been checked by computer that the elements $(w,\Phi)$ such that $|I_{\Phi'}(w|_{\Phi'})|\neq 2$ for all such subsystems $\Phi'$ (that is, the elements in $W(\Phi)$ avoiding the patterns in question) all have a pivot, and are thus separable.

\subsubsection{Type $G_2$}
Let $\Phi$ denote the root system of type $G_2$, with simple roots $\alpha_1$ and $\alpha_2$ (see Figure \ref{fig:G2}).  Since any separable element must have a pivot, it is clear that the only possible inversion sets $I_{\Phi}(w)$ with $(w,\Phi)$ separable are $\emptyset, \{\alpha_1\}, \{\alpha_2\}$ and their complements in $\Phi^+$.  Thus the six elements with inversions sets of sizes two, three, and four are minimal non-separable elements, agreeing with Theorem \ref{thm:pattern-classification}.

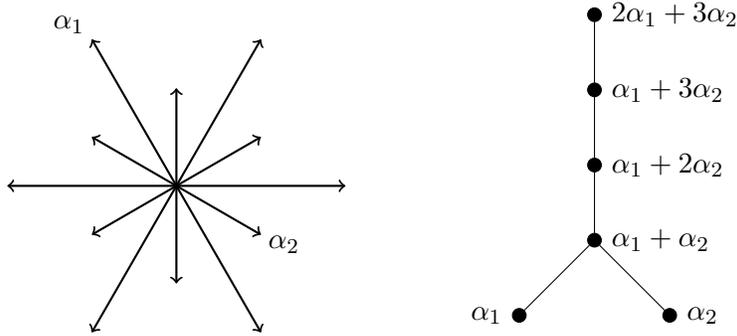
\begin{figure}[ht]
    \begin{tikzpicture}[scale=1.3]
    \draw[thick, ->] (0,0)--(1.73,0);
    \draw[thick, ->] (0,0)--(.865,1.5);
    \draw[thick, ->] (0,0)--(-.865,1.5);
    \draw[thick, ->] (0,0)--(-1.73,0);
    \draw[thick, ->] (0,0)--(-.865,-1.5);
    \draw[thick, ->] (0,0)--(.865,-1.5);
    
    \draw[thick, ->] (0,0)--(.865,.5);
    \draw[thick, ->] (0,0)--(0,1);
    \draw[thick, ->] (0,0)--(-.865,.5);
    \draw[thick, ->] (0,0)--(-.865,-.5);
    \draw[thick, ->] (0,0)--(0,-1);
    \draw[thick, ->] (0,0)--(.865,-.5);
    
    \node at (-1.1,1.65) {$\alpha_1$};
    \node at (1.1,-.6) {$\alpha_2$};
    \end{tikzpicture} \hspace{0.5in}
    \begin{tikzpicture}
    \node[draw,shape=circle,fill=black,scale=0.5](a)[label=left: {$\alpha_1$}] at (-1,0) {};
    \node[draw,shape=circle,fill=black,scale=0.5](b)[label=right: {$\alpha_2$}] at (1,0) {};
    \node[draw,shape=circle,fill=black,scale=0.5](c)[label=right: {$\alpha_1+\alpha_2$}] at (0,1) {};
    \node[draw,shape=circle,fill=black,scale=0.5](d)[label=right: {$\alpha_1+2\alpha_2$}] at (0,2) {};
    \node[draw,shape=circle,fill=black,scale=0.5](e)[label=right: {$\alpha_1+3\alpha_2$}] at (0,3) {};
    \node[draw,shape=circle,fill=black,scale=0.5](f)[label=right: {$2\alpha_1+3\alpha_2$}] at (0,4) {};
    
    \draw (a) -- (c) -- (d) -- (e) -- (f);
    \draw (b) -- (c);
    \end{tikzpicture}
    \caption{The root system of type $G_2$ and its root poset.}
    \label{fig:G2}
\end{figure}

\subsection{Completing the proof}

We now complete the proof of Theorem \ref{thm:pattern-classification}.

\begin{proof}[Proof of Theorem \ref{thm:pattern-classification}]
Let $w$ be an element in the Weyl group of an arbitrary finite root system $\Phi$, and suppose that $(w,\Phi)$ avoids the patterns from Theorem \ref{thm:pattern-classification}.  Then, decomposing $\Phi$ as a direct some of irreducible root systems $\Phi_1 \oplus \cdots \oplus \Phi_k$ it is clear that $(w|_{\Phi_i}, \Phi_i)$ also avoids these patterns for all $i$.  By the results of Sections \ref{sec:simply-laced} and \ref{sec:not-simply-laced}, each $(w|_{\Phi_i},\Phi_i)$ is therefore separable, and thus $(w,\Phi)$ is separable.  The reverse direction is given by Proposition \ref{prop:characterized-by-avoidance}.  For the claim in Theorem \ref{thm:pattern-classification} about simply laced types, just note that no element of a Weyl group of simply-laced type may contain one of the forbidden patterns of type $B_2$ or $G_2$, since all roots have the same length.
\end{proof}

\section*{Acknowledgements}
The authors wish to thank Alexander Postnikov, Victor Reiner, and Anders Bj\"{o}rner for helpful comments.

\bibliographystyle{plain}
\bibliography{arxiv-v2}
\end{document}